\documentclass[a4paper]{amsart}

\usepackage[utf8]{inputenc} % set input encoding (not needed with XeLaTeX)
\usepackage[british]{babel} %otherwise words break in horrible places
\usepackage{culmus} %hebrew letters
\usepackage{dsfont}

\usepackage{hyperref}

\begin{hyphenrules}{british}
\end{hyphenrules}

\usepackage{graphicx} % support the \includegraphics command and options
\usepackage{booktabs} % for much better looking tables
\usepackage{array} % for better arrays (eg matrices) in maths
\usepackage{paralist} % very flexible & customisable lists (eg. enumerate/itemize, etc.)
\usepackage{verbatim} % adds environment for commenting out blocks of text & for better verbatim
\usepackage{subfig} % make it possible to include more than one captioned figure/table in a single float
\usepackage{scalerel,stackengine} %per usare \reallywidehat
\usepackage{bmpsize}
\usepackage{amssymb}
\usepackage{mathtools}
\usepackage{amsmath}	
\usepackage{amssymb}
\usepackage{amsthm}
\usepackage[mathcal]{eucal}
\usepackage{graphicx}
\usepackage{faktor} %"diagonal" quotient style
\usepackage{mathrsfs} % calligraphy letters, other style
\usepackage[all, cmtip]{xy} % good looking and flexible matrices
\usepackage{systeme}
\usepackage[titletoc]{appendix}
\usepackage{leftindex}
\usepackage{enumitem}
\usepackage{appendix}
\usepackage{yhmath} %Really big hats using \widehat

\theoremstyle{plain}
\newtheorem{theorem}[subsection]{Theorem}

\newtheorem{proposition}[subsection]{Proposition}
\newtheorem{corollary}[subsection]{Corollary}

\theoremstyle{definition}
\newtheorem{definition}[subsection]{Definition}

\theoremstyle{remark}

\newtheorem{examples}[subsection]{Examples}

\stackMath
\newcommand\reallywidehat[1]{%
\savestack{\tmpbox}{\stretchto{%
  \scaleto{%
    \scalerel*[\widthof{\ensuremath{#1}}]{\kern-.6pt\bigwedge\kern-.6pt}%
    {\rule[-\textheight/2]{1ex}{\textheight}}%WIDTH-LIMITED BIG WEDGE
  }{\textheight}%
}{0.5ex}}%
\stackon[1pt]{#1}{\tmpbox}%
}
\newcommand{\pullback}[4]{{#1 \leftindex_{#2}{\times}_{#3} #4}}

\usepackage{color}

\newcommand{\E}{\mathrm{E}}
\newcommand{\I}{\mathrm{I}}
\newcommand{\X}{\mathbb{X}}
\newcommand{\Y}{\mathbb{Y}}

\DeclareSymbolFont{alephbet}{HE8}{frank}{m}{n}
\SetSymbolFont{alephbet}{bold}{HE8}{frank}{b}{n}
\DeclareMathSymbol{\samech}{\mathord}{alephbet}{"F1}
\DeclareMathSymbol{\mem}{\mathord}{alephbet}{"EE}

\usepackage{chngcntr}
\counterwithin*{equation}{section}

\def\pullback{% with thanks to Valerian Even
 \ar@{-}[]+R+<4pt,-3pt>;[]+RD+<4pt,-6pt>%
 \ar@{-}[]+D+<3pt,-6pt>;[]+RD+<4pt,-6pt>}
 
\def\pushout{%
  \ar@{-}[]+U+<-2pt,5pt>;[]+LU+<-4pt,5pt>% vertical line up-left
  \ar@{-}[]+L+<-4pt,2pt>;[]+LU+<-4pt,5pt>}% horizontal line left-up     

\begin{document}

\title[Actions, semidirect products and crossed semimodules$\dots$]{Actions, semidirect products and crossed semimodules in the category of small categories with a fixed set of objects}

\author[S.~Ambra]{Stefano Ambra}
\address[Stefano Ambra]{Dipartimento di Matematica ``Federigo Enriques'', Universit\`{a} degli Studi di Milano, Via Saldini 50, 20133 Milano, Italy}
\email{stefano.ambra@unimi.it}

\begin{abstract}
We generalize to the fibres of the fibration $\mathcal{O}\colon\mathbf{Cat}\rightarrow\mathbf{Set},$ defined by mapping a small category $\mathbb{X}$ to its set of objects $X_0=ob(\mathbb{X}),$ the classical notions of action and semidirect product of monoids. We prove that the equivalence between monoid actions of a monoid $Y$ and Schreier split extensions on $Y,$ which is well known to generalize the equivalence between actions and split extensions for groups, is an instance of a broader adjunction between Schreier points and actions in the fibres $\mathcal{O}^{-1}(B).$ This adjunction is an equivalence if and only if $B=1,$ i.e., for the category $\mathbf{Mon}$ of monoids. Similarly, we prove that there is an adjunction (which, in the case of monoids, results in a known equivalence due to Patchkoria) between Schreier internal categories in the fibres $\mathcal{O}^{-1}(B)$ and the category of crossed semimodules in $\mathcal{O}^{-1}(B).$ The latter are defined by translating in $\mathcal{O}^{-1}(B)$ the notion of crossed semimodule in $\mathbf{Mon}.$ Eventually, we prove that, by defining crossed modules appropriately, this last adjunction yields an equivalence between crossed modules and Schreier internal groupoids in the fibres of $\mathcal{O}.$
\end{abstract}


\maketitle
%----------------------------------------------------------------------------------------------------------------------
\section{Introduction}

%------------------------------------------------------------------------------------------------------------------------
Let $\mathbf{Cat}$ denote the category of small categories and functors. The \emph{set-of-objects} functor
\begin{equation}
\label{eqn:O}
\mathcal{O}\colon \mathbf{Cat}\longrightarrow\mathbf{Set}, \ \mathbb{X}\mapsto ob(\mathbb{X}),
\end{equation}
is then a Grothendieck fibration, whose fibres $\mathcal{O}^{-1}(B)\coloneqq\mathbf{Cat}_B$ are Barr-exact \cite{barr} (see~\cite{tour}), while the category $\mathbf{Cat}$ itself is not even regular (see for example~\cite{foundations} for a discussion on epimorphisms in $\mathbf{Cat}$).

When $B=1$ is a terminal object in $\mathbf{Set},$ the fibre $\mathbf{Cat}_1$ is equivalent to the category $\mathbf{Mon}$ of monoids and monoid homomorphisms, where the classical notions of action and semidirect product are available. 

Building on these notions, A. Patchkoria proved in~\cite{P-semi} that the well-known equivalence $\mathbf{Xmod}\cong\mathrm{Cat}(\mathbf{Gp})$ between crossed modules and internal categories in the category of groups extends to an equivalence $\mathbf{Xsmod}\cong\mathrm{SCat}(\mathbf{Mon})$ between crossed \emph{semimodules} (a generalization to monoids of the notion of crossed module) and \emph{Schreier} internal categories in $\mathbf{Mon},$ i.e. internal categories
\begin{equation*}
\xymatrix{{\mathbb{X}:X_2\cong{X_1\times_{X_0}X_1}} \ar[r]^-{m} &{X_1} \ar@<.9ex>[r]^-{d} \ar@<-.9ex>[r]_-{c} &{X_0} \ar[l]|{u} }
\end{equation*}
in $\mathbf{Mon}$ with the following factorization property: If $\xymatrixcolsep{1.5pc}\xymatrix{k\colon K(d) \ar@{>->}[r] &X_1}$ is a fixed kernel of $d,$ for every $x\in X_1$ there exists a unique $a\in K(d)$ such that $x=k(a)+ud(x)$ ($+$ denoting the monoid operation on $X_1$). 

A pair $(f,s)$ of monoid homomorphisms $\xymatrix{X \ar@<.5ex>[r]^-f  &Y \ar@<.5ex>[l]^-s}$ such that $fs=1_Y$ and satisfying the same factorization property of the pair $(d,u)$ above is called a \emph{Schreier point} in $\mathbf{Mon}$ (on $Y$), and one can prove that such Schreier points correspond to monoid actions (namely, monoid homomorphisms $Y\rightarrow\mathrm{End}(X)$ of a monoid $Y$ into the monoid of endomorphisms of a monoid $X$), in the same way as split short exact sequences correspond to actions for groups (and indeed for any semi-abelian category, see \cite{bourn-janelidze}).

Moreover, the category of monoids, which is neither protomodular \cite{bourn-proto} nor Mal'tsev \cite{carboni-lambeck-pedicchio}, is in fact \emph{relatively} so (viz., it is $S$-protomodular in the sense of \cite{S-protomodular}) with respect to the class of Schreier points.

Now, Schreier points can be defined in every fibre $\mathbf{Cat}_B$ of the functor \eqref{eqn:O}, and it is possible to prove that for any set $B$ the category $\mathbf{Cat}_B$ is $S$-protomodular, and hence $S$-Mal'tsev (see \cite{mal'tsev-reflection}), with respect to this class of points. 

A first easy observation, which is the starting point for Section \ref{sec:actions} below, is that, all the same, for any set $B$ and every pair of objects $\mathbb{X}, \mathbb{Y}\in\mathbf{Cat}_B,$ one can consider a notion of action $\alpha$ of $\mathbb{Y}$ on $\mathbb{X}$ in $\mathbf{Cat}_B,$ and form a semidirect product $\mathbb{X}\rtimes_{\alpha}\mathbb{Y}$ of $\mathbb{X}$ and $\mathbb{Y}$ accordingly. In the case $B=1,$ both the action and the semidirect product so introduced coincide with the usual ones for monoids, but it is to be remarked that, contrary to the case of monoids, when $|B|\geq 2$ and $\mathbb{X},$ $\mathbb{Y}$ are not totally disconnected, the set of arrows of the semidirect product $\mathbb{X}\rtimes_{\alpha}\mathbb{Y}$ is different from the set of arrows of the binary product $\mathbb{X}\times\mathbb{Y}$ in $\mathbf{Cat}_B.$ 

It is then the main result of Section \ref{sec:actions} that the equivalence between Schreier points on a monoid $Y$ and monoid actions of $Y$ is an instance of a more general adjunction 
\begin{equation*}
\xymatrix{{SPt_{\mathbb{Y}}} \ar@/^0.7pc/[r]  \ar@{}[r]|{\bot } &{\mathrm{Act}_{\mathbb{Y}}}\ar@/^0.7pc/[l]  }
\end{equation*}
between Schreier points on $\mathbb{Y}$ in $\mathbf{Cat}_B$ and actions in $\mathbf{Cat}_B$ of the small category $\mathbb{Y}.$ This adjunction is an equivalence if and only if $B=1.$

Similarly, using these notions of action and semidirect product, in Section \ref{sec:crossed} we generalize to $\mathbf{Cat}_B,$ for an arbitrary set $B,$ the definition of crossed semimodule of \cite{JS, P-semi}. Again, the main result here is that there is an adjunction 
\begin{equation*}
\xymatrix{ {\mathrm{SCat}(\mathbf{Cat}_B)} \ar@/^0.7pc/[r]  \ar@{}[r]|{\bot} &{\mathbf{Xsmod}(\mathbf{Cat}_B)} \ar@/^0.7pc/[l]  }
\end{equation*}
between Schreier internal categories in $\mathbf{Cat}_B$ and the category of crossed semimodules in $\mathbf{Cat}_B,$ which is an equivalence only in the case of monoids. 

The restriction of the left adjoint $\mathrm{SCat}(\mathbf{Cat}_B)\longrightarrow\mathbf{Xsmod}(\mathbf{Cat}_B)$ to the full subcategory $\mathrm{SGpd}(\mathbf{Cat}_B)\subseteq\mathrm{SCat}(\mathbf{Cat}_B)$ of Schreier internal groupoids in $\mathbf{Cat}_B$ yields an equivalence, for every set $B,$ between the latter and the crossed semimodules $\delta\colon\mathbb{X}\rightarrow\mathbb{Y}$ whose domain is a totally disconnected groupoid. Accordingly, we call such a $\delta$ a crossed \emph{module} in $\mathbf{Cat}_B.$

The motivating reason for our study of these notions in the fibres of \eqref{eqn:O}, as it will be shown in a further paper \cite{ADM}, is that they allow for a categorical description in terms of \emph{direction functors} \cite{baer-sums, aspherical, bourn-rodelo} of the cohomology groups $H^n(\mathbb{Y},A)$ of a small category $\mathbb{Y}\in\mathbf{Cat}_B$ with coefficients in a $\mathbb{Y}$-module $A$ (that is, a functor $A\colon\mathbb{Y}\rightarrow\mathbf{Ab}$ whose codomain is the category of abelian groups), in the sense of \cite{cocat, golasinski}. 
%----------------------------------------------------------------------------------------------------------------------------
\section{Actions and Schreier points in $\mathbf{Cat}_B$}
\label{sec:actions}
Let $B$ be a set. As in the Introduction, we shall denote by $\mathbf{Cat}_B$ the fibre $\mathcal{O}^{-1}(B)$ of the set-of-objects functor \eqref{eqn:O}. A morphism $\mathbb{X}\rightarrow\mathbb{Y}$ in $\mathbf{Cat}_B$ is then a functor
\begin{equation}
\label{eqn:morphism}
\begin{aligned}
{
\xymatrixcolsep{0.5pc}
\xymatrix{
{\mathbb{X}} \ar[d] &{:}\\
{\mathbb{Y}} &{:}
}
}
\xymatrix{
{X_2} \ar[d]_-{F_2} \ar[r]^-{m_X} &{X_1} \ar[d]_-{F_1} \ar@<.9ex>[r]^-{d_X} \ar@<-.9ex>[r]_-{c_X} &{B} \ar[l]|-{u_X} \ar@{=}[d] \\
{Y_2}  \ar[r]_-{m_Y} &{Y_1} \ar@<.9ex>[r]^-{d_Y} \ar@<-.9ex>[r]_-{c_Y} &{B} \ar[l]|-{u_Y}
}
\end{aligned}
\end{equation}
 which is the identity on objects. The map $F_2\colon X_2\cong X_1\times_{B}X_1\rightarrow Y_1\times_BY_1\cong Y_2$ between the sets of composable arrows is given by $(x_1,x_2)\mapsto\big(F_1(x_1),F_1(x_2)\big),$ so that \eqref{eqn:morphism} is completely determined by the map $F_1\colon X_1\rightarrow Y_1$ between the sets of arrows. Accordingly, we shall simply use the notation $F$ both for the functor $F\colon\mathbb{X}\rightarrow\mathbb{Y}$ and the map $F=F_1\colon X_1\rightarrow Y_1.$
 
Given a morphism $F\colon\X\rightarrow\Y$ in $\mathbf{Cat}_B,$ we shall denote by $\mathbb{K}(F)$ the pullback
\[ \xymatrix{
\mathbb{K}(F) \pullback \ar[d] \ar@{>->}[r] &{\X} \ar[d]^-{F} \\
\Delta_B \ar@{>->}[r]_-{u_Y} &{\Y,}
} \]
where $\Delta_B$ is the discrete category on $B$ and $u_Y$ is the units map of $\Y.$ (It can be realized so as to have the disjoint union $K(F)\coloneqq\bigsqcup_{b\in B}F^{-1}(1_b)$ as set of arrows.)
 
Moreover, for a small category $\xymatrix{{\mathbb{X}:X_2} \ar[r]^-{m} &{X_1} \ar@<.9ex>[r]^-{d} \ar@<-.9ex>[r]_-{c} &{B,} \ar[l]|-{u}}$ we shall often use the notation $u(b)=1_b$ to denote the identity arrow of $b\in B,$ and the notation $x_2\cdot x_1\colon a\rightarrow c$ to denote the composition $m(x_1,x_2)$ of $x_1\colon a\rightarrow b$ and $x_2\colon b\rightarrow c$ in $\mathbb{X}.$ 

Now, classically, when $B=1$ and $X,Y\in\mathbf{Cat}_1\cong\mathbf{Mon}$ are monoids, an action of $Y$ on $X$ is a monoid homomorphism $Y\rightarrow\mathrm{End}(X)$ between $Y$ and the monoid $\mathrm{End}(X)$ of monoid endomorphisms of $X.$

When $\mathbb{X}\in\mathbf{Cat}_B$ for an arbitrary set $B,$ we can easily mimic the definition of $\mathrm{End}(X)$ as follows. For all $a,b\in B,$ define
\[
\mathbb{X}(a,b)\coloneqq\{x\in X_1:d(x)=a \ \& \ c(x)=b\};
\]
it follows that, for any $b\in B,$ the set $\mathbb{X}(b,b)$ is a monoid under the composition of $\mathbb{X},$ with neutral element $1_b,$ and we can consider the set
\[
\mathrm{E}_{\mathbb{X}}(a,b)\coloneqq\mathbf{Mon}\big(\mathbb{X}(a,a),\mathbb{X}(b,b)\big)
\]
of all monoid homomorphisms with domain $\mathbb{X}(a,a)$ and codomain $\mathbb{X}(b,b).$ 

Consider then the small category 
\[
\xymatrix{{\mathrm{E}(\mathbb{X}):\E_2(\mathbb{X})} \ar[r]^-{m_E} &{\E_1(\mathbb{X})} \ar@<.9ex>[r]^-{d_E} \ar@<-.9ex>[r]_-{c_E} &{B} \ar[l]|-{u_E}}
\] 
whose arrows $f\in E_1(\mathbb{X})$ with domain $d_E(f)=a$ and codomain $c_E(f)=b$ are the elements of the set $\mathrm{E}_{\mathbb{X}}(a,b).$ The composition in $\mathrm{E}(\mathbb{X}),$ where defined, is given by the composition of morphisms in $\mathbf{Mon},$ and units are defined by the identity morphisms $u_E(b)=id_{\X(b,b)}\in \mathrm{E}_{\mathbb{X}}(b,b).$ 

Define also the subcategory 
\begin{equation}
\label{eqn:cat_I}
\xymatrix{{\mathrm{I}(\mathbb{X}):\I_2(\mathbb{X})} \ar[r]^-{m} &{\I_1(\mathbb{X})} \ar@<.9ex>[r]^-{d} \ar@<-.9ex>[r]_-{c} &{B} \ar[l]|-{u}}
\end{equation}
of $\mathbb{X},$ whose arrows are given by the disjoint union $\I_1(\mathbb{X})\coloneqq\bigsqcup_{b\in B}\mathbb{X}(b,b).$

Observe that $\I(\mathbb{X})$ is totally disconnected by construction (that is, no arrow $a\rightarrow b$ exists in $\I(\mathbb{X})$ if $a\neq b$), and that one has $\I(\mathbb{X})=\mathbb{X}$ if and only if $\mathbb{X}$ is totally disconnected. On the contrary, $\E(\mathbb{X})$ is never totally disconnected if $|B|\geq 2,$ not even when $\mathbb{X}$ is, because for any $a\neq b$ in $B$ one can consider the trivial monoid homomorphism $\mathbb{X}(a,a)\rightarrow \mathbb{X}(b,b),$ $x\mapsto 1_b.$ Moreover, it is clear that $\E(\mathbb{X})=\E\big(\I(\mathbb{X})\big),$ and that for $B=1$ (in which case $\mathbb{X}$ corresponds to an ordinary monoid $X$) the category $\E(\mathbb{X})$ corresponds to the monoid $\mathrm{End}(X).$ 
\begin{definition}
\label{def:action}
If $\mathbb{X}, \mathbb{Y}\in\mathbf{Cat}_B,$ an \emph{action} of $\mathbb{Y}$ on $\mathbb{X}$ is a morphism $\alpha\colon\mathbb{Y}\rightarrow\E(\mathbb{X})$ in $\mathbf{Cat}_B.$ 
\end{definition}
\begin{examples}
\label{ex:actions}
\begin{enumerate}
\item When $\mathbb{Y}$ is totally disconnected, one can always consider the trivial action $\alpha(y\colon b\rightarrow b)=u_E(b).$ 
\item Let $B=\{a,b\}$ and define $\mathbb{Y}$ as the category having as arrows the identity maps $1_a,$ $1_b$ and a unique non-trivial arrow $\mathbb{Y}(a,b)=\{\ast\}.$ Then, for any $\mathbb{X}\in\mathbf{Cat}_B$ and any fixed $f\in\mathbf{Mon}\big(\mathbb{X}(a,a),\mathbb{X}(b,b)\big),$ the map $\alpha(\ast)=f$ gives an action of $\mathbb{Y}$ on $\mathbb{X}.$
\end{enumerate}
\end{examples}
We can now define:
\begin{definition}
Let $\alpha\colon\mathbb{Y}\rightarrow\E(\mathbb{X})$ be an action in $\mathbf{Cat}_B.$ The \emph{semidirect product} $\mathbb{X}\rtimes_{\alpha}\mathbb{Y}$ in $\mathbf{Cat}_B$ is the small category 
\begin{equation}
\label{eqn:semidirect}
\xymatrix{{\mathbb{X}\rtimes_{\alpha}\mathbb{Y}:P_2(\mathbb{X},\mathbb{Y})} \ar[r]^-{m_P} &{P_1(\mathbb{X},\mathbb{Y})} \ar@<.9ex>[r]^-{d_P} \ar@<-.9ex>[r]_-{c_P} &{B} \ar[l]|-{u_P}}
\end{equation} 
where:
\begin{itemize}
\item an arrow $a\xrightarrow{g}b\in P_1(\mathbb{X},\mathbb{Y})$ with domain $d_P(g)=a$ and codomain $c_P(g)=b$ is a pair $g=(x,y)$ with $y\in\mathbb{Y}(a,b)$ and $x\in\mathbb{X}(b,b);$
\item $u_P(b)=(1_b,1_b);$
\item the composition of $(b\xrightarrow{x}b,a\xrightarrow{y}b)$ followed by $(c\xrightarrow{x'}c,b\xrightarrow{y'}c)$ is given by
\begin{equation*}
(x',y')\cdot(x,y)=m_P\big((x,y),(x',y')\big)\coloneqq \big( x'\cdot \alpha(y')(x), y'\cdot y\big).
\end{equation*}
\end{itemize}
\end{definition}
Clearly, when $B=1$ this reduces to the ordinary semidirect product of monoids. 

Observe also that if $\mathbb{Y}$ is totally disconnected and the action $\alpha\colon\mathbb{Y}\rightarrow\E(\mathbb{X})$ is trivial (in the sense of Example \ref{ex:actions}(1)), then the semidirect product $\mathbb{X}\rtimes_{\alpha}\mathbb{Y}$ coincides with the binary product $\mathbb{X}\times\mathbb{Y}$ in $\mathbf{Cat}_B,$ but if both $\mathbb{X}$ and $\mathbb{Y}$ are not totally disconnected then the set $P_1(\mathbb{X},\mathbb{Y})$ is a proper subset of the set of arrows of $\mathbb{X}\times\mathbb{Y}.$

Now, a pair of morphisms $\xymatrix{{\mathbb{X}} \ar@<.5ex>[r]^-F &{\mathbb{Y}} \ar@<.5ex>[l]^-S}$ in $\mathbf{Cat}_B$ such that $FS=id_{\mathbb{Y}}$ (also referred to as a \emph{point} on $\mathbb{Y}$ in $\mathbf{Cat}_B$) is called a \emph{Schreier point} if for every arrow $x\in\mathbb{X}(a,b)$ there exists a unique $k\in\mathbb{X}(b,b)$ such that $F(k)=1_b$ and $x=k\cdot SF(x).$ (This notion was first considered in \cite{mal'tsev-reflection}, under the name of \emph{functors with cofibrant splittings}.)

For a fixed $\Y\in\mathbf{Cat}_B,$ we shall denote by $SPt_{\Y}$ the category of all Schreier points in $\mathbf{Cat}_B,$ whose morphisms $\big(\xymatrix{{\mathbb{X}} \ar@<.5ex>[r]^-F &{\mathbb{Y}} \ar@<.5ex>[l]^-S}\big)\rightarrow \big(\xymatrix{{\mathbb{X}'} \ar@<.5ex>[r]^-{F'} &{\mathbb{Y}} \ar@<.5ex>[l]^-{S'}}\big)$ are the morphisms $G\colon\X\rightarrow\X'$ in $\mathbf{Cat}_B$ such that the diagram
\begin{equation}
\label{eqn:mor_points}
\begin{aligned}
\xymatrixcolsep{1.5pc}
\xymatrix{
{\X} \ar@<-.5ex>[rd]_-{F} \ar[rr]^-G &\ &{\X'} \ar@<-.5ex>[ld]_-{F'}  \\
&{\Y} \ar@<-.5ex>[lu]_-{S} \ar@<-.5ex>[ru]_-{S'} &\
}
\end{aligned}
\end{equation}
is commutative (in the sense that both $F'G=F$ and $GS=S'$).

We shall make use of the following property: Given a commutative diagram
\[ 
\xymatrix{
{\mathbb{K}(F)} \ar[d]_-{F_1} \ar@{>->}[r] &{\X}\ar[d]_-{F_2} \ar@<.5ex>[r]^-{F} \ar@<-.5ex>[r]_-{S} &{\Y} \ar[d]^-{F_3} \\ 
{\mathbb{K}(F')} \ar@{>->}[r] &{\X'} \ar@<.5ex>[r]^-{F'} \ar@<-.5ex>[r]_-{S'} &{\Y'}
}
\]
in $\mathbf{Cat}_B$ in which $(F,S)$ and $(F',S')$ are Schreier points and $F_1$ and $F_3$ are isomorphisms, the morphism $F_2$ is also an isomorphism. With a slight abuse of language (since $\mathbf{Cat}_B$ is not pointed), we express this fact by saying that the Split Short Five Lemma for Schreier points holds in $\mathbf{Cat}_B.$

Indeed, one can prove that the category $\mathbf{Cat}_B$ is $S$-protomodular with respect to the class of Schreier points, in the sense of \cite{mal'tsev-reflection}, Definition 4.3 (which is just the non-pointed version of \cite{S-protomodular}, Definition 3.1, where this notion was introduced): see \cite{mal'tsev-reflection}, Theorem 3.1 and Propositions 3.4, 3.6, for a proof. The fact that the Split Short Five Lemma for Schreier points holds in $\mathbf{Cat}_B$ is then a consequence of \cite{mal'tsev-reflection}, Theorem 4.1(2) (which is, again, the non-pointed version of \cite{S-protomodular}, Proposition 3.2$(iii)$).

Schreier points admit an equivalent description in terms of $\mathbf{Mon}$-valued functors:
\begin{proposition}[cf.~\cite{mal'tsev-reflection}, Lemma 3.1]
\label{prop:equivalence_schreier_points_functor}
The category $SPt_{\Y}$ of Schreier points on $\Y$ in $\mathbf{Cat}_B$ is equivalent to the functor category $\mathrm{Fun}(\Y,\mathbf{Mon}).$
\end{proposition}
This equivalence may be described in terms of semidirect products in $\mathbf{Cat}_B,$ as follows. 

Given a functor $F\colon\Y\rightarrow\mathbf{Mon},$ denote by $F^+$ (using the notation of \cite{cocat, golasinski}) the small, totally disconnected category on $B$ whose set of arrows is the disjoint union $\bigsqcup_{b\in B}F(b).$ (Of course, an arrow $z\in F(b)$ has the element $b$ both as domain and codomain, and units in $F^+$ are given by the neutral elements of the monoids $F(b).$) 

Then, for every such functor $F,$ there is an action $\alpha_F$ of $\Y$ on $F^+$ given by 
\[ \alpha_F(a\xrightarrow{y}b)\coloneqq F(y)\colon F(a)\rightarrow F(b),\]
and in the Schreier point $\xymatrix{{\mathbb{X}} \ar@<.5ex>[r]^-F &{\mathbb{Y}} \ar@<.5ex>[l]^-S}$ corresponding to $F$ in the equivalence of Proposition \ref{prop:equivalence_schreier_points_functor} one has $\X\cong F^+\rtimes_{\alpha_F}\Y.$

Conversely, if $\alpha\colon\Y\rightarrow\E(\X)$ is an action in $\mathbf{Cat}_B,$ one obtains a functor $F\colon\Y\rightarrow\mathbf{Mon}$ by setting $F(b)\coloneqq\X(b,b)$ and $F(a\xrightarrow{y}b)\coloneqq\alpha(y)\colon F(a)\rightarrow F(b).$ 

Thus, we may equivalently define an action of $\Y$ on $\X$ in $\mathbf{Cat}_B$ as a functor $F\colon\Y\rightarrow\mathbf{Mon}$ such that $F^+=\I(\X).$

We are ready to state the main result of this section.

Denote by $\mathrm{Act}_{\Y},$ again for a fixed object $\Y\in\mathbf{Cat}_B,$ the category of all pairs $(\X,\alpha)$ where $\X\in\mathbf{Cat}_B$ and $\alpha\colon\Y\rightarrow\E(\X)$ is an action, whose morphisms $(\X,\alpha)\rightarrow(\mathbb{U},\vartheta)$ are the morphisms $H\colon\X\rightarrow\mathbb{U}$ in $\mathbf{Cat}_B$ such that for all $y\in\Y(a,b)$ the square
\[
\xymatrix{
\X(a,a) \ar[d]_-H \ar[r]^-{\alpha(y)} &{\X(b,b)} \ar[d]^-H \\
\mathbb{U}(a,a) \ar[r]_-{\vartheta(y)} &{\mathbb{U}(b,b)}
}
\]
is commutative.
\begin{theorem}
\label{prop:adjunction_actions}
The functor
\begin{equation}
\label{eqn:K} \mathcal{K}\colon SPt_{\Y} \longrightarrow \mathrm{Act}_{\Y}, \ (\xymatrix{{\mathbb{X}} \ar@<.5ex>[r]^-F &{\mathbb{Y}} \ar@<.5ex>[l]^-S})\longmapsto(\mathbb{K}(F),\omega_{(F,S)}), 
\end{equation}
where the action $\omega_{(F,S)}\colon\Y\rightarrow\E\big(\mathbb{K}(F)\big)$ is given by 
\begin{equation*}
\omega_{(F,S)}(a\xrightarrow{y}b)\colon F^{-1}(1_a)\rightarrow F^{-1}(1_b), \ k\mapsto{^yk} \ \text{ s.t. } \ S(y)\cdot k={^yk}\cdot S(y),
\end{equation*}
is left adjoint to the functor
\begin{equation} 
\label{eqn:P}
\mathcal{P}\colon \mathrm{Act}_{\Y}\longrightarrow SPt_{\Y} , (\X,\alpha)\longmapsto (\xymatrix{{\mathbb{X}\rtimes_{\alpha}\mathbb{Y}} \ar@<.5ex>[r]^-{\Pi} &{\mathbb{Y}} \ar@<.5ex>[l]^-{\Sigma}}),
\end{equation}
where $\Pi$ is the projection morphism $\Pi(x,y)\coloneqq y$ and $\Sigma(a\xrightarrow{y}b)\coloneqq(1_b,y).$
\end{theorem}
\begin{proof}
The functor \eqref{eqn:K} operates on morphisms by mapping a morphism \eqref{eqn:mor_points} of points on $\Y$ to the restriction $\mathcal{K}(G)\colon(\mathbb{K}(F),\omega_{(F,S)})\rightarrow(\mathbb{K}(F'),\omega_{(F',S')})$ given by $\mathcal{K}(G)\colon\bigsqcup_{b\in B}F^{-1}(1_b)\rightarrow\bigsqcup_{b\in B}F'^{-1}(1_b), \ k\mapsto G(k).$ It is well defined, by $F'G=F,$ and it preserves the action by $S'=G S$ and the uniqueness of $^{y}G(k)$ in the equality $S'(y)G(k)={^{y}G(k)}S'(y),$ for any $y\in\Y(a,b).$

Next, observe that the functor \eqref{eqn:P} is well defined on the objects, meaning that $\xymatrix{{\mathbb{X}\rtimes_{\alpha}\mathbb{Y}} \ar@<.5ex>[r]^-{\Pi} &{\mathbb{Y}} \ar@<.5ex>[l]^-{\Sigma}}$ is indeed a Schreier point in $\mathbf{Cat}_B,$ by $(b\xrightarrow{x}b,a\xrightarrow{y}b)=(x,1_b)\cdot(1_b,y).$ Moreover, we have $\mathbb{K}(\Pi)\cong\I(\X),$ in the notation of \eqref{eqn:cat_I}. On the morphisms, $\mathcal{P}$ is defined by mapping an arrow $H\colon(\X,\alpha)\rightarrow(\mathbb{U},\vartheta)$ in $\mathrm{Act}_{\Y}$ to the morphism 
\[ \xymatrixcolsep{1.5pc}
\xymatrix{
{\X\rtimes_{\alpha}\Y} \ar@<-.5ex>[rd]_-{\Pi} \ar[rr]^-{\mathcal{P}(H)} &\ &{\mathbb{U}\rtimes_{\vartheta}\Y} \ar@<-.5ex>[ld]_-{\Pi'}  \\
&{\Y} \ar@<-.5ex>[lu]_-{\Sigma} \ar@<-.5ex>[ru]_-{\Sigma'} &\
} \]
given by $\mathcal{P}(H)\colon P_1(\X,\Y)\rightarrow P_1(\mathbb{U},\Y), \ (x,y)\mapsto(H(x),y).$

The composite functor $\mathcal{P}\mathcal{K}\colon SPt_{\Y}\longrightarrow SPt_{\Y}$ then sends a Schreier point $\xymatrix{{\mathbb{X}} \ar@<.5ex>[r]^-F &{\mathbb{Y}} \ar@<.5ex>[l]^-S}$ to the Schreier point $\xymatrix{{\mathbb{K}(F)\rtimes_{\omega_{(F,S)}}\Y} \ar@<.5ex>[r]^-{\Pi} &{\mathbb{Y},} \ar@<.5ex>[l]^-{\Sigma}}$ yielding a morphism 
\[
\xymatrix{
{\mathbb{K}(F)} \ar@{>->}[r] \ar@{=}[d] &{\mathbb{K}(F)\rtimes_{\omega_{(F,S)}}\Y} \ar[d]^-{\beta_{(F,S)}} \ar@<.5ex>[r]^-{\Pi} &{\mathbb{Y}} \ar@{=}[d] \ar@<.5ex>[l]^-{\Sigma} \\
{\mathbb{K}(F)} \ar@{>->}[r] &{\X} \ar@<.5ex>[r]^-{F} &{\mathbb{Y}, \ar@<.5ex>[l]^-{S}}
}
\]
$\beta_{(F,S)}\colon P_1(\mathbb{K}(F),\Y)\rightarrow X_1,$ $(b\xrightarrow{x}b,a\xrightarrow{y}b)\mapsto m_X(S(y),x)=x\cdot S(y).$ 

By the Split Short Five Lemma for Schreier points in $\mathbf{Cat}_B,$ $\beta_{(F,S)}$ is an isomorphism, and we denote its inverse by 
\begin{equation}
\label{eqn:unit_1}\eta_{(F,S)}\coloneqq\beta_{(F,S)}^{-1}\colon (\xymatrix{{\mathbb{X}} \ar@<.5ex>[r]^-F &{\mathbb{Y}} \ar@<.5ex>[l]^-S}) \xlongrightarrow{\sim}\mathcal{P}\mathcal{K}(\xymatrix{{\mathbb{X}} \ar@<.5ex>[r]^-F &{\mathbb{Y}} \ar@<.5ex>[l]^-S}).
\end{equation}

Similarly, the composition $\mathcal{K}\mathcal{P}\colon\mathrm{Act}_{\Y}\longrightarrow\mathrm{Act}_{\Y}$ maps $\big(\X,\alpha\colon\Y\rightarrow\E(\X)\big)$ to $\big(\I(\X),\alpha\colon\Y\rightarrow\E(\I(\X))=\E(\X)\big),$ i.e. to the same action of $\Y,$ but seen as an action on the totally disconnected subcategory $\I(\X)\subseteq\X.$ Denote by 
\begin{equation}
\label{eqn:counit_1}
\xymatrix{ {\varepsilon_{(\X,\alpha)}\colon\mathcal{K}\mathcal{P}(\X,\alpha)=(\I(\X),\alpha)} \ar@{>->}[r] &{(\X,\alpha)}}
\end{equation}
the inclusion morphism. 

Then the bijection defined by mapping a morphism 
\[ H\colon \mathcal{K}(\xymatrix{{\mathbb{X}} \ar@<.5ex>[r]^-F &{\mathbb{Y}} \ar@<.5ex>[l]^-S})=(\mathbb{K}(F),\omega_{(F,S)})\rightarrow(\mathbb{U},\vartheta)\]
in $\mathrm{Act}_{\Y}$ to the morphism of Schreier points 
\[ \xymatrixcolsep{1.5pc}
\xymatrix{
{\X} \ar@<-.5ex>[rd]_-{F} \ar[rr]^-{\Phi(H)} &\ &{\mathbb{U}\rtimes_{\vartheta}\Y} \ar@<-.5ex>[ld]_-{\Pi}  \\
&{\Y,} \ar@<-.5ex>[lu]_-{S} \ar@<-.5ex>[ru]_-{\Sigma} &\
} \]
where $\Phi(H)\colon X_1\rightarrow P_1(U,Y),$ $\big(x=k\cdot SF(x)\big)\mapsto\big(H(k),F(x)\big),$ and conversely by sending a morphism of Schreier points
\[ \xymatrixcolsep{1.5pc}
\xymatrix{
{\X} \ar@<-.5ex>[rd]_-{F} \ar[rr]^-{G} &\ &{\mathbb{U}\rtimes_{\vartheta}\Y} \ar@<-.5ex>[ld]_-{\Pi}  \\
&{\Y} \ar@<-.5ex>[lu]_-{S} \ar@<-.5ex>[ru]_-{\Sigma} &\
} \]
given by $G\colon X_1\rightarrow P_1(U,Y),$ $x\mapsto G(x)=\big(\widetilde{G}(x),F(x)\big)$ for some $\widetilde{G}\colon\X\rightarrow\mathbb{U},$ to the restriction
\[ \widetilde{G}\colon (\mathbb{K}(F),\omega_{(F,S)})\rightarrow (\mathbb{U},\vartheta), \ (x\in F^{-1}(1_b))\mapsto \widetilde{G}(x), \]
determines an adjunction $\mathcal{K}\dashv\mathcal{P}$ whose unit and counit are given by the morphisms $\eta_{(F,S)}$ and $\varepsilon_{(\X,\alpha)},$ respectively.
\end{proof}
Observe that, the left adjoint $\mathcal{K}$ being fully faithful (the unit \eqref{eqn:unit_1} is an isomorphism), we have:
\begin{corollary}
The category $SPt_{\Y}$ is equivalent to the coreflective subcategory $(\mathrm{Act}_{\Y})_{t.d.}\subseteq\mathrm{Act}_{\Y}$ whose objects are the pairs $(\X,\alpha\colon\Y\rightarrow\E(\X)\big)$ with $\X$ totally disconnected.
\end{corollary}
Indeed, if $\X$ is totally disconnected, then the counit \eqref{eqn:counit_1} is also, trivially, an isomorphism.
%----------------------------------------------------------------------------------------------------------------------------
\section{Crossed semimodules and crossed modules in $\mathbf{Cat}_B$}
\label{sec:crossed}
Armed with Definition \ref{def:action}, we can transplant in $\mathbf{Cat}_B$ the notion of crossed semimodule, introduced for monoids in \cite{JS} (and later studied in \cite{P-semi}) as a generalization of the ordinary notion of crossed modules.
\begin{definition}
\label{def:crossed_semimodule}
A \emph{crossed semimodule} in $\mathbf{Cat}_B,$ for an arbitrary set $B,$ is a pair $(\X\xrightarrow{\delta}\Y,\alpha)$ where $\X,\Y\in\mathbf{Cat}_B,$ $\delta$ is a morphism in $\mathbf{Cat}_B$ and $\alpha\colon\Y\rightarrow\E(\X)$ is an action, satisfying the following axioms for all $a,b\in B$:
\begin{enumerate}
\item $\delta\big(\alpha(y)(x)\big)\cdot y=y\cdot\delta(x)$ (composition in $\Y$) for all $y\in\Y(a,b)$ and $x\in\X(a,a);$ 
\item $\alpha\big(\delta(w)(x)\big)\cdot w=w\cdot x$ (composition in $\X$) for all $w\in\X(a,b)$ and $x\in\X(a,a).$
\end{enumerate}
A morphism 
\[ (\X\xlongrightarrow{\delta}\Y,\alpha)\longrightarrow(\X'\xlongrightarrow{\delta'}\Y',\alpha')\]
of crossed semimodules is a pair $(\varphi\colon\X\rightarrow\X',\psi\colon\Y\rightarrow\Y')$ of morphisms in $\mathbf{Cat}_B$ such that the square
\begin{equation}
\label{eqn:mor_semimod}
\begin{aligned}
\xymatrix{
{\X} \ar[d]_-{\varphi} \ar[r]^-{\delta} &{\Y} \ar[d]^-{\psi} \\
{\X'} \ar[r]_-{\delta'} &{\Y'}
}
\end{aligned}
\end{equation}
commutes, and satisfying for every $y\in\Y(a,b)$ and $x\in\X(a,a)$ (for any $a,b\in B$) the equality
\[ \varphi\big(\alpha(y)(x)\big)=\alpha'\big(\psi(y)\big)\big(\varphi(x)\big). \]
\end{definition}
We shall denote the category of crossed semimodules in $\mathbf{Cat}_B$ by $\mathbf{Xsmod}_B.$

When $B=1$ (in which case $\mathbf{Xsmod}_1$ is denoted simply by $\mathbf{Xsmod}$), it is proven in \cite{P-semi} that one has an equivalence $\mathbf{Xsmod}\cong\mathrm{SCat}(\mathbf{Mon}),$ where the latter denotes the category of Schreier internal categories in $\mathbf{Mon},$ i.e. the full subcategory of $\mathrm{Cat}(\mathbf{Mon})$ of all internal categories
\begin{equation*}
\xymatrix{{X_2\cong{X_1\times_{X_0}X_1}} \ar[r]^-{m} &{X_1} \ar@<.9ex>[r]^-{d} \ar@<-.9ex>[r]_-{c} &{X_0} \ar[l]|{u} }
\end{equation*}
such that $(d,u)$ is a Schreier point. 

Of course, since the notion of Schreier point is available in $\mathbf{Cat}_B$ for any set $B,$ we can more generally consider the category $\mathrm{SCat}(\mathbf{Cat}_B)$ of Schreier internal categories in $\mathbf{Cat}_B.$ Unsurprisingly, as was for Theorem \ref{prop:adjunction_actions}, the above equivalence is an instance of a more general adjunction:
\begin{theorem}
\label{prop:adjunction_Xsmod_SCat}
There is an adjunction $\mathcal{L}\dashv\mathcal{S}$ between the functor
\begin{equation}\label{eqn:L} \mathcal{L}\colon\mathrm{SCat}(\mathbf{Cat}_B)\longrightarrow\mathbf{Xsmod}_B, \end{equation}
defined by mapping a Schreier internal category
\begin{equation}\label{eqn:Sch_cat} \xymatrix{{\X_2\cong{\X_1\times_{\X_0}\X_1}} \ar[r]^-{M} &{\X_1} \ar@<.9ex>[r]^-{D} \ar@<-.9ex>[r]_-{C} &{\X_0} \ar[l]|{U} } \end{equation}
to the restriction $\delta\coloneqq C_|\colon \mathbb{K}(D)\rightarrow\X_0,$ together with the induced action $\omega_{(D,U)}$ such that $\big(\mathbb{K}(D),\omega_{(D,U)}\big)=\mathcal{K}(D,U)$ as in \eqref{eqn:K},
and the functor
\begin{equation}\label{eqn:S} \mathcal{S}\colon \mathbf{Xsmod}_B\longrightarrow\mathrm{SCat}(\mathbf{Cat}_B)\end{equation}
which sends a crossed semimodule $(\X\xlongrightarrow{\delta}\Y,\alpha)$ to 
\begin{equation}\label{eqn:semi-cat}  \xymatrix{{C(\X,\Y)} \ar[r]^-{\Xi} &{\X\rtimes_{\alpha}\Y} \ar@<.9ex>[r]^-{\Pi} \ar@<-.9ex>[r]_-{\Lambda} &{\Y,} \ar[l]|-{\Sigma} } \end{equation}
where $(\xymatrix{{\mathbb{X}\rtimes_{\alpha}\mathbb{Y}} \ar@<.5ex>[r]^-{\Pi} &{\mathbb{Y}} \ar@<.5ex>[l]^-{\Sigma}})=\mathcal{P}(\X,\alpha)$ as in \eqref{eqn:P}, $\Lambda$ is defined on $P_1(\X,\Y)$ (as in \eqref{eqn:semidirect}) by $\Lambda(b\xrightarrow{x}b,a\xrightarrow{y}b)\coloneqq\delta(x)\cdot y,$ and the composition $\Xi$ is given on the pullback
\begin{equation}
\label{eqn:C(X,Y)}
\begin{aligned}
\xymatrix{
{C(\X,\Y)} \pullback \ar[r] \ar[d] &{\X\rtimes_{\alpha}\Y} \ar[d]^-{\Pi} \\
{\X\rtimes_{\alpha}\Y} \ar[r]_-{\Lambda} &{\Y}
}
\end{aligned}
\end{equation}
(so that the set of arrows of $C(\X,\Y)$ can be realized as the set of triples $(z,x,y)$ with $y\in\Y(a,b)$ and $x,z\in\X(b,b)$)  by 
\[ \Xi(z,x,y)\coloneqq (z\cdot x,y).\]
\end{theorem}
\begin{proof}
\emph{Step} $1$: The functor \eqref{eqn:L} is well defined.

Consider a Schreier internal category \eqref{eqn:Sch_cat} in $\mathbf{Cat}_B,$ and define $\delta\colon \mathbb{K}(D)\rightarrow \X_0$ as the restriction of $C$ to $\mathbb{K}(D).$ By definition of $\mathcal{K}$ \eqref{eqn:K}, the $\X_0$-action $\omega_{(D,U)}$ on $\mathbb{K}(D)$ is given by 
\[ \omega_{(D,U)}\colon D^{-1}(1_a)\rightarrow D^{-1}(1_b), \ k\mapsto {^yk} \ s.t. \ U(y)\cdot k={^yk}\cdot U(y)\]
for every arrow $y\in\X_0(a,b)$ (composition in $\X_1$), and we must show that the pair $(\mathbb{K}(D)\xrightarrow{\delta} \X_0,\omega_{(D,U)})$ satisfies the two axioms of Definition \ref{def:crossed_semimodule}. 
\begin{enumerate}
\item $\delta\big(\omega_{(D,U)}(y)(k)\big)\cdot y=y\cdot \delta(k)$ \emph{(composition in $\X_0$) for every} $y\in\X_0(a,b)$ \emph{and} $k\in D^{-1}(1_a).$

Indeed, we have
\begin{equation*}
\begin{split}
\delta\big(\omega_{(D,U)}(y)(k)\big)\cdot y&=C({^yk})\cdot y\\
&=C\big({^yk}\cdot U(y)\big) \\
&=C\big(U(y)\cdot k\big)\\
&=y\cdot C(k)\\
&=y\cdot\delta(k).
\end{split}
\end{equation*}
\item $\omega_{(D,U)}\big(\delta(x)\big)(k)\cdot x=x\cdot k$ \emph{(composition in $\X_1$) for every} $x,k\in D^{-1}(1_a).$

I.e., we have to show that ${^{C(x)}k}\cdot x=x\cdot k.$ To prove this, observe that since $(D,U)$ is a Schreier point, then for every pair $(z,v)$ of arrows in $\X_1$ such that $D(v)=C(z)$ (which implies that $v,z\colon a\rightarrow b$ are parallel arrows in $\X_1,$ since $D$ and $C$ are morphisms in $\mathbf{Cat}_B$) we have with respect to the composition of $\X_1$
\[ z=w\cdot UD(z)\]
and
\begin{equation*}\begin{split}v=w'\cdot UD(v)=w'\cdot UC(z)&=w'\cdot UC(w)\cdot UCUD(z)\\&=w'\cdot UC(w)\cdot UD(z) \end{split}\end{equation*} 
for two unique arrows $w,w'\in D^{-1}(1_b).$

Then
\begin{equation}
\begin{aligned}
\begin{split}
M(z,v)&=M\big(w\cdot UD(z),w'\cdot UC(w)\cdot UD(z)\big)\\
&=M\Big((1_b,w')\circ\big(w,UC(w)\big)\circ\big(UD(z),UD(z)\big)\Big)  \\
&=M\Big(\big(UD(w'),w'\big)\circ\big(w,UC(w)\big)\circ\big(UD(z),UD(z)\big)\Big)\\
&=M\big(UD(w'),w'\big)\cdot M\big(w,UC(w)\big)\cdot M\big(UD(z),UD(z)\big)\\
&=w'\cdot w\cdot UD(z), \label{eqn:star}
\end{split}
\end{aligned}
\end{equation}
where $\circ$ denotes the composition in $\mathbb{X}_2$ (which is component-wise the composition of $\X_1$). 

It follows that, for every $x,k\in D^{-1}(1_b)$:
\begin{equation*}
\begin{split}
^{C(x)}k\cdot x&={^{C(x)}k}\cdot x\cdot \underbrace{UD(x)}_{=1_b}\\
&=M\big(x\cdot UD(x),\underbrace{{^{C(x)}k}\cdot UC(x)}_{=UC(x)\cdot k}\cdot UD(x) \big) \ (\text{by} \ \eqref{eqn:star}) \\
&=M\big(x\cdot UD(x),UC(x)\cdot k\cdot UD(x) \big)\\
&=M\big(x\cdot UD(k),UC(x)\cdot k\big) \\
&=M\Big(\big(x,UC(x)\big)\circ\big(UD(k),k\big)\Big) \\
&=M\big(x,UC(x)\big)\cdot M\big(UD(k),k\big)\\
&=x\cdot k,
\end{split}
\end{equation*}
which is the desired equality.
\end{enumerate}
Thus, $\mathcal{L}$ is well defined on the objects. On the morphisms, we define $\mathcal{L}$ by mapping an internal functor
\begin{equation}\label{eqn:mor_SCat}\begin{aligned} \xymatrix{
{\X_2} \ar[d]_-{F_2} \ar[r]^-{M} &{\X_1} \ar[d]_-{F_1} \ar@<.9ex>[r]^-{D} \ar@<-.9ex>[r]_-{C} &{\X_0} \ar[l]|-{U} \ar[d]^-{F_0} \\
{\X'_2}  \ar[r]_-{M'} &{\X'_1} \ar@<.9ex>[r]^-{D'} \ar@<-.9ex>[r]_-{C'} &{\X'_0} \ar[l]|-{U'}
} \end{aligned}\end{equation}
in $\mathrm{SCat}(\mathbf{Cat}_B)$ to the diagram
\[ \xymatrixcolsep{2.5pc}\xymatrix{
{\mathbb{K}(D)} \ar[d]_-{{F_1}_|} \ar[r]^-{\delta=C_|} &{\X_0} \ar[d]^-{F_0} \\
{\mathbb{K}(D')}  \ar[r]_-{\delta'=C'_|} &{\X'_0} 
} \]
given by the restriction of $F_1$ to $\mathbb{K}(D).$ (It is well defined and commutative by the commutativity of \eqref{eqn:mor_SCat}, and it is easy to see that it is indeed a morphism in $\mathbf{Xsmod}_B.$)\\

\emph{Step} $2$: The functor \eqref{eqn:S} is well defined.

We already know that the pair $(\Pi,\Sigma)$ in \eqref{eqn:semi-cat} is a Schreier point, and the only non-trivial point in the verification that $\Lambda(x,y)=\delta(x)\cdot y$ is a morphism  in $\mathbf{Cat}_B$ is that it preserves the composition of $\X\rtimes_{\alpha}\Y$: 
\begin{equation*}
\begin{split}
\Lambda\big((x',y')\cdot(x,y)\big)&= \Lambda\big(x'\cdot \alpha(y')(x),y'\cdot y\big)\\
&=\delta\big(x'\cdot\alpha(y')(x)\big)\cdot y'\cdot y \\
&=\delta(x')\cdot \underbrace{\delta\big(\alpha(y')(x)\big)\cdot y'}_{=y'\cdot \delta(x)}\cdot y\\
&=\delta(x')\cdot y'\cdot\delta(x)\cdot y \\
&=\Lambda(x',y')\cdot\Lambda(x,y),
\end{split}
\end{equation*}
using Definition \ref{def:crossed_semimodule}(1).

Next, the pullback \eqref{eqn:C(X,Y)} defining $C(\X,\Y)$ is constructed by taking the following level-wise pullbacks in $\mathbf{Set}$:
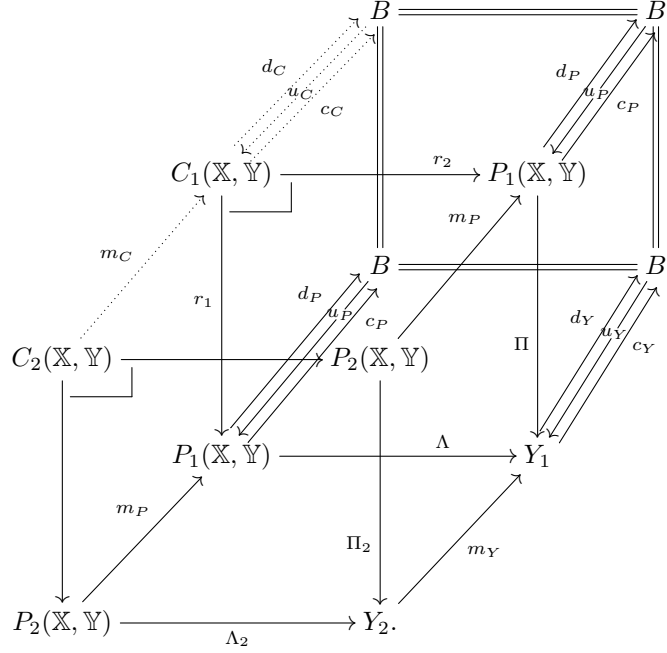
\[ 
\xymatrixcolsep{1.3pc}
\xymatrixrowsep{1.7pc}
\xymatrix{
&\ &{B} \ar@{.>}[ldd]|{u_C} \ar@{=}[ddd] \ar@{=}[rr] &\ &{B} \ar[ldd]|{u_P} \ar@{=}[ddd] \\
&\ &\ &\ &\ \\
&{C_1(\X,\Y)} \ar@{.>}@<.99ex>[ruu]^-{d_C} \ar@{.>}@<-.99ex>[ruu]_-{c_C} \pullback\ar[ddd]_(.45){r_1}  \ar[rr]^(.70){r_2} &\  &{P_1(\X,\Y)} \ar@<.99ex>[ruu]^-{d_P} \ar@<-.99ex>[ruu]_-{c_P} \ar[ddd]_(.60){\Pi} &\ \\
&\ &{B}  \ar[ldd]|(.25){u_P} \ar@{=}[rr] &\ &{B} \ar[ldd]|(.35){u_Y} \\
{C_2(\X,\Y)} \ar@{.>}[ruu]^-{m_C} \ar[ddd] \pullback \ar[rr] &\  &{P_2(\X,\Y)} \ar[ruu]^(.70){m_P} \ar[ddd]_(.70){\Pi_2} &\ &\ \\
&{P_1(\X,\Y)}  \ar@<.99ex>[ruu]^(.75){d_P} \ar@<-.99ex>[ruu]_(.80){c_P} \ar[rr]^(.70){\Lambda} &\  &{Y_1} \ar@<.99ex>[ruu]^(.65){d_Y} \ar@<-.99ex>[ruu]_(.65){c_Y} &\ \\
&\ &\ &\ &\ \\
{P_2(\X,\Y)} \ar[ruu]^(.60){m_P}  \ar[rr]_-{\Lambda_2} &\ &{Y_2.} \ar[ruu]_-{m_Y} &\ &\ 
}
\]
Therefore, one can describe $C_1(\X,\Y)$ as the set of triples $(z,x,y)$ with $y\in\Y(a,b)$ and $x,z\in\X(b,b)$ and projections $r_1(z,x,y)=(x,y)$ and $r_2(z,x,y)=\big(z,C(x)\cdot y\big),$ the set $C_2(\X,\Y)$ as the set of sextuples $(z,x,y,z',x',y')$ such that $y\in\Y(a,b),$ $y'\in\Y(b,c),$ $z.x\in\X(b,b),$ and $z',x'\in\X(c,c),$ and realize the composition $m_C$ by
\begin{equation*}
\begin{split}
m_C(z,x,y,z',x',y')&\coloneqq(z',x',y')\cdot(z,x,y)\\
&\coloneqq \Big( z'\cdot \alpha\big(\delta(x')\cdot y'\big)(z),x'\cdot\alpha(y')(x),y'\cdot y\Big).
\end{split}
\end{equation*}
Again, the only non-trivial point, in order to verify that \eqref{eqn:semi-cat} is indeed a (Schreier) internal category in $\mathbf{Cat}_B,$ is that the morphism $\Xi\colon C(\X,\Y)\rightarrow\X\rtimes_{\alpha}\Y$ defined by $\Xi(z,x,y)\coloneqq(z\cdot x,y)$ preserves the composition $m_C.$ To see this, we compute:
\begin{align*}
\begin{split}
\bullet \ \ \Xi\big((z',x',y')\cdot(z,x,y)\big) &=\Xi\Big( z'\cdot \alpha\big(\delta(x')\cdot y'\big)(z),x'\cdot\alpha(y')(x),y'\cdot y\Big)\\
&=\Big( z'\cdot \underbrace{\alpha\big(\delta(x')\cdot y'\big)(z)\cdot x'}_{A}\cdot\alpha(y')(x),y'\cdot y\Big),
\end{split}
\\
\begin{split}
\bullet \ \ \Xi(z',x',y')\cdot\Xi(z,x,y)&=(z'\cdot x',y')\cdot(z\cdot x,y)  \\
&=\big( z'\cdot x'\cdot \alpha(y')(z\cdot x),y'\cdot y \big) \\
&=\big( z'\cdot \underbrace{x'\cdot \alpha(y')(z)}_{B}\cdot \alpha(y')(x),y'\cdot y \big), 
\end{split}
\end{align*}
and the two expressions are equal by Definition \ref{def:crossed_semimodule}(2):
\begin{equation*}
\underbrace{\alpha\big(\delta(x')\cdot y'\big)(z)\cdot x'}_{A}=\alpha\big(\delta(x')\big)\big(\alpha(y')(z)\big)\cdot x' = \underbrace{x'\cdot \alpha(y')(z)}_{B}.
\end{equation*}
On the arrows, $\mathcal{S}$ is defined by mapping a morphism \eqref{eqn:mor_semimod} of crossed semimodules to 
\[ 
\xymatrix{
{\X\rtimes_{\alpha}\Y} \ar[d]_-{\varphi\rtimes\psi} \ar@<.99ex>[r]^-{\Pi} \ar@<-.99ex>[r]_-{\Lambda} &{\Y} \ar[l]|(.40){\Sigma} \ar[d]^-{\psi} \\ 
{\X'\rtimes_{\alpha}\Y'} \ar@<.99ex>[r]^-{\Pi'} \ar@<-.99ex>[r]_-{\Lambda'} &{\Y'.} \ar[l]|(.36){\Sigma'}
}
\]

\emph{Step} $3$: The adjunction $\mathcal{L}\dashv\mathcal{S}.$

The composite functor $\mathcal{S}\mathcal{L}$ maps a Schreier internal category
\[ \xymatrix{{\mathfrak{X}:\X_2} \ar[r]^-{M} &{\X_1} \ar@<.9ex>[r]^-{D} \ar@<-.9ex>[r]_-{C} &{\X_0} \ar[l]|{U} } \]
to
\[ \xymatrix{{C(\mathbb{K}(D),\X_0)} \ar[r]^-{\Xi} &{\mathbb{K}(D)\rtimes_{\omega_{(D,U)}}\X_0} \ar@<.9ex>[r]^-{\Pi} \ar@<-.9ex>[r]_-{\Lambda} &{\X_0.} \ar[l]|-{\Sigma} } \]

Observe that, as for the adjunction $\mathcal{K}\dashv\mathcal{P}$ of Theorem \ref{prop:adjunction_actions}, we have a morphism $\zeta_{\mathfrak{X}}$ in $\mathbf{Cat}_B$ making the following diagram commute
\[ 
\xymatrix{
{\mathbb{K}(D)} \ar@{=}[d] \ar@{>->}[r] &{\mathbb{K}(D)\rtimes_{\omega_{(D,U)}}\X_0} \ar[d]^-{\zeta_{\mathfrak{X}}} \ar@<.9ex>[r]^-{\Pi} \ar@<-.9ex>[r]_-{\Lambda} &{\X_0} \ar[l]|-{\Sigma} \ar@{=}[d] \\
{\mathbb{K}(D)} \ar@{>->}[r] &{\X_1} \ar@<.9ex>[r]^-{D} \ar@<-.9ex>[r]_-{C} &{\X_0,} \ar[l]|-{U}
}
\]
given by $\zeta_{\mathfrak{X}}\colon P_1(\mathbb{K}(D),\X_0)\rightarrow X_1=arr(\X_1),$ $(k,x)\mapsto k\cdot U(x)$ (composition in $\X_1$).

By the Split Short Five Lemma for the class of Schreier points, $\zeta_{\mathfrak{X}}$ is an isomorphism in $\mathbf{Cat}_B.$

To prove that it is actually a morphism between $\mathcal{S}\mathcal{L}(\mathfrak{X})$ and $\mathfrak{X}$ in $\mathrm{SCat}(\mathbf{Cat}_B),$ we show that it preserves the composition $\Xi.$ Indeed, for every $(z,w,x)\in C_1(\mathbb{K}(D),\X_0)$:
\[ \zeta_{\mathfrak{X}}\big(\Xi(z,w,x)\big)=\zeta_{\mathfrak{X}}(z\cdot w,x)=z\cdot w\cdot U(x)\]
and
\begin{equation*}
\begin{split}
M\big(\zeta_{\mathfrak{X}}(w,x),\zeta_{\mathfrak{X}}(z,C(w)\cdot x)\big)&=M\big(w\cdot U(x),z\cdot UC(w)\cdot U(x)\big)\\
&=M\big(w\cdot UDU(x),z\cdot UC(w)\cdot U(x)\big)\\
&=z\cdot w \cdot U(x),
\end{split}
\end{equation*}
using \eqref{eqn:star}. 

Define the inverse of $\zeta_{\mathfrak{X}}$ as
\begin{equation}
\label{eqn:unit_2}
\eta_{\mathfrak{X}}\coloneqq \zeta_{\mathfrak{X}}^{-1}\colon \mathfrak{X} \longrightarrow \mathcal{S}\mathcal{L}(\mathfrak{X}).
\end{equation}
Similarly, it is not difficult to see that $\mathcal{L}\mathcal{S}(\X\xlongrightarrow{\delta}\Y,\alpha)=(\I(\X)\xlongrightarrow{\delta_|}\Y,\alpha),$ so that we have an inclusion morphism 
\begin{equation}\label{eqn:counit_2}\xymatrix{ {\varepsilon_{(\delta,\alpha)}\colon\mathcal{L}\mathcal{S}(\X\xlongrightarrow{\delta}\Y,\alpha)}\ar@{>->}[r] &{(\X\xlongrightarrow{\delta}\Y,\alpha).}}\end{equation}

Then, the bijection defined by mapping a morphism $(\varphi,\psi)\colon\mathcal{L}(\mathfrak{X})\rightarrow(\X'\xrightarrow{\delta'}\Y',\alpha')$
\[ \xymatrixcolsep{2.5pc}\xymatrix{
{\mathbb{K}(D)} \ar[d]_-{\varphi} \ar[r]^-{\delta=C_|} &{\X_0} \ar[d]^-{\psi} \\
{\mathbb{X}'}  \ar[r]_-{\delta'} &{\Y'} 
} \]
in $\mathbf{Xsmod}_B$ to 
\[ 
\xymatrix{
{\X_1} \ar[d]_-{\widehat{(\varphi,\psi)}} \ar@<.99ex>[r]^-{D} \ar@<-.99ex>[r]_-{C} &{\X_0} \ar[l]|-{U} \ar[d]^-{\psi} \\ 
{\X'\rtimes_{\alpha}\Y'} \ar@<.99ex>[r]^-{\Pi'} \ar@<-.99ex>[r]_-{\Lambda'} &{\Y',} \ar[l]|(.40){\Sigma'}
}
\]
where $\widehat{(\varphi,\psi)}\colon X_1\rightarrow P_1(\X',\Y'),$ $\big(x=k\cdot UD(x)\big)\mapsto \big(\varphi(k),\psi D(x)\big),$ whose inverse associates a morphism $(F,G)\colon \mathfrak{X}\rightarrow \mathcal{S}(\delta',\alpha')$
\[ \xymatrix{
{\X_1} \ar[d]_-{F=\langle\widetilde{F},GD\rangle} \ar@<.99ex>[r]^-{D} \ar@<-.99ex>[r]_-{C} &{\X_0} \ar[l]|-{U} \ar[d]^-{G} \\ 
{\X'\rtimes_{\alpha}\Y'} \ar@<.99ex>[r]^-{\Pi'} \ar@<-.99ex>[r]_-{\Lambda'} &{\Y',} \ar[l]|(.40){\Sigma'}
}\]
of Schreier internal categories with
\[ \xymatrixcolsep{2.5pc}\xymatrix{
{\mathbb{K}(D)} \ar[d]_-{\widetilde{F}} \ar[r]^-{C_|} &{\X_0} \ar[d]^-{G} \\
{\mathbb{X}'}  \ar[r]_-{\delta'} &{\Y',} 
} \]
yields the desired adjunction $\mathcal{L}\dashv\mathcal{S},$ having the morphisms \eqref{eqn:unit_2} as unit and \eqref{eqn:counit_2} as counit.
\end{proof}
Again, observe that since \eqref{eqn:unit_2} is an isomorphism, it follows that the category of Schreier internal categories in $\mathbf{Cat}_B$ is equivalent to a coreflective subcategory of $\mathbf{Xsmod}_B,$ namely its full subcategory of crossed semimodules $(\X\xlongrightarrow{\delta}\Y,\alpha)$ with $\X$ totally disconnected.

Now, suppose that 
\[ \xymatrix{{\mathfrak{G}:\X_2} \ar[r]^-{M} &{\X_1}  \ar@(dl,dr)[]_-I  \ar@<.9ex>[r]^-{D} \ar@<-.9ex>[r]_-{C} &{\X_0} \ar[l]|{U} } \]
is a Schreier internal \emph{groupoid} in $\mathbf{Cat}_B.$ Then $\mathbb{K}(D)$ (defined, as above, so to have the disjoint union $\bigsqcup_{b\in B}D^{-1}(1_b)$ as set of arrows) is a groupoid on $B.$ Indeed, let $b\in B$ and fix an arrow $k\in D^{-1}(1_b).$ By the assumption, there exists $x=I(k)\in\X(b,b)$ such that $D(x)=C(k),$ $C(x)=D(k)$ and $M(k,x)=UD(k)=1_b.$ Since $(D,U)$ is a Schreier point in $\mathbf{Cat}_B,$ we can write
\[ x=k'\cdot UD(x)=k'\cdot UC(k)\]
for a unique $k'\in D^{-1}(1_b),$ and by \eqref{eqn:star} we have:
\begin{equation*}
\begin{split}
1_b=M(k,x)&=M\big(k,k'\cdot UC(k)\big) \\
&=M\big(k\cdot \underbrace{UD(k)}_{=1_b},k'\cdot UC(k)\cdot \underbrace{UD(k)}_{=1_b}\big) \\
&=k'\cdot k\cdot UD(k)\\
&=k'\cdot k.
\end{split}
\end{equation*}
Thus, any arrow $k\in D^{-1}(1_b)$ has a left inverse in $D^{-1}(1_b),$ which entails that the monoid $\big(D^{-1}(1_b),\cdot, 1_b\big)$ is a group. 

The converse is also true:
\begin{proposition}
If $(\X\xlongrightarrow{\delta}\Y,\alpha)$ is a crossed semimodule such that the subcategory $\I(\X)$ (as in \eqref{eqn:cat_I}) is a groupoid on $B,$ then the Schreier internal category $\mathcal{S}(\X\xlongrightarrow{\delta}\Y,\alpha)$ is an internal groupoid in $\mathbf{Cat}_B.$
\end{proposition}
\begin{proof}
Define $I\colon \X\rtimes_{\alpha}\Y\rightarrow \X\rtimes_{\alpha}\Y$ by 
\[ I\colon P_1(\X,\Y)\rightarrow P_1(\X,\Y), \ (b\xrightarrow{x}b,a\xrightarrow{y}b)\mapsto \big(x^{-1},\Lambda(x,y)=\delta(x)\cdot y\big),\]
where $x^{-1}$ denotes the inverse of $x$ in $\X(b,b).$

The only non trivial verification to prove that $I$ is a morphism in $\mathbf{Cat}_B$ is that it preserves the composition $m_P$ of \eqref{eqn:semidirect}:
\begin{equation*}
\begin{split}
I\big( (x',y')\cdot(x,y)\big)&=I\big(x'\cdot\alpha(y')(x),y'\cdot y\big)\\
&=\Big(\big(x'\cdot\alpha(y')(x)\big)^{-1},\delta\big(x'\cdot\alpha(y')(x)\big)\cdot y'\cdot y\Big)\\
&=\Big(\alpha(y')(x^{-1})\cdot (x')^{-1},\delta(x')\cdot\underbrace{\delta\big(\alpha(y')(x)\big)\cdot y'}_{=y'\cdot \delta(x)}\cdot y\Big)\\
&=\big(\alpha(y')(x^{-1})\cdot (x')^{-1},\delta(x')\cdot y'\cdot \delta(x)\cdot y\big);
\end{split}
\end{equation*}
at the same time
\begin{equation*}
\begin{split}
I(x',y')\cdot I(x,y)&=\big((x')^{-1},\delta(x')\cdot y'\big)\cdot\big(x^{-1},\delta(x)\cdot y\big) \\
&=\Big( (x')^{-1}\cdot\alpha\big(\delta(x')\cdot y'\big)(x^{-1}),\delta(x')\cdot y'\cdot \delta(x)\cdot y\Big),
\end{split}
\end{equation*}
and indeed $ (x')^{-1}\cdot\alpha\big(\delta(x')\cdot y'\big)(x^{-1})=\alpha(y')(x^{-1})\cdot (x')^{-1}$ because
\begin{equation*}
\begin{split}
\alpha\big(\delta(x')\cdot y'\big)(x^{-1})\cdot x' &=\alpha\big(\delta(x')\big)\big(\alpha(y')(x^{-1})\big)\cdot x'\\
&=x'\cdot \alpha(y')(x^{-1})
\end{split}
\end{equation*}
by Definition \ref{def:crossed_semimodule}$(2).$

Moreover, we have for every $(x,y)\in P_1(\X,\Y)$
\[ \Pi\big(I(x,y)\big)=\Pi\big(x^{-1},\Lambda(x,y)\big)=\Lambda(x,y) \]
and 
\[ \Lambda\big(I(x,y)\big)=\Lambda\big(x^{-1},\delta(x)\cdot y\big)=\delta(x^{-1})\cdot\delta(x)\cdot y=y=\Pi(x,y). \]
Finally, the composition of $(b\xrightarrow{x}b,a\xrightarrow{y}b)$ followed by $I(x,y)=\big(x^{-1},\delta(x)\cdot y\big)$ with respect to the composition law $\Xi$ of \eqref{eqn:semi-cat} is given by
\[ \Xi(x^{-1},x,y)=(x^{-1}\cdot x,y)=(1_b,y)=\Sigma(y)=\Sigma\Pi(x,y),\]
and the composition of $I(x,y)$ followed by $(x,y)=\big(x,\delta(x^{-1})\cdot\delta(x)\cdot y\big)$ is given by
\[ \Xi\big(x,x^{-1},\delta(x)\cdot y\big)=\big(x\cdot x^{-1},\delta(x)\cdot y\big)=\big(1_b,\Lambda(x,y)\big)=\Sigma\Lambda(x,y).\]
\end{proof}
All this motivates the ensuing:
\begin{definition}
A \emph{crossed module} in $\mathbf{Cat}_B$ is a crossed semimodule $(\X\xlongrightarrow{\delta}\Y,\alpha)$ in which $\X$ is a totally disconnected groupoid on $B.$
\end{definition}
Denote by $\mathbf{Xmod}_B\subseteq\mathbf{Xsmod}_B$ the full subcategory of crossed modules in $\mathbf{Cat}_B.$

Then, by Theorem \ref{prop:adjunction_Xsmod_SCat}, we have at once:
\begin{corollary}
The adjunction $\mathcal{L}\dashv\mathcal{S}$ yields an equivalence of categories
\begin{equation}\label{eqn:groupoids_equiv} \mathrm{SGpd}(\mathbf{Cat}_B)\cong \mathbf{Xmod}_B\end{equation}
between Schreier internal groupoids and crossed modules in $\mathbf{Cat}_B.$
\end{corollary}
This extends to $\mathbf{Cat}_B,$ for an arbitrary set $B,$ the analogous result of \cite{P-semi} for the category of monoids. 

It is possible to show that, by using the equivalence \eqref{eqn:groupoids_equiv}, one can describe the third cohomology group $H^3(\Y,A)$ of a small category $\Y$ with coefficients in a $\Y$-module $A\colon\Y\rightarrow\mathbf{Ab},$ as introduced in \cite{cocat}, by means of the direction functor \cite{aspherical, bourn-rodelo}. A description of the same cohomology group in terms of crossed extensions is outlined in \cite{golasinski}. The detailed study of $H^3(\Y,A),$ as well as of the other cohomology groups $\{H^n(\Y,A)\}_{n\geq 2},$ in terms of the above mentioned direction functors, will be carried out in \cite{ADM}.
%------------------------------------------------------------------------------------------------------------------------
\section*{Acknowledgments}
The author would like to thank Andrea Montoli for his mentorship and the valuable discussions on the subject.

The author is a member of the Gruppo Nazionale per le Strutture Algebriche, Geometriche e le loro Applicazioni (GNSAGA) dell'Istituto Nazionale di Alta Matematica ``Francesco Severi''.
%------------------------------------------------------------------------------------------------------------------------

\end{document}